\documentclass[11pt,a4paper]{article}
\usepackage[top=2.5cm,bottom=2.5cm,left=2.2cm,right=2.2cm]{geometry}% Add by K
\usepackage[T1]{fontenc}
\usepackage[utf8]{inputenc}
\usepackage{color}
\usepackage{graphicx}
\usepackage{amsfonts}
\usepackage{extarrows}
\usepackage{amsmath,amsthm,amssymb,color}
\usepackage{hyperref}
\usepackage{eepic}
\usepackage{lineno}
\usepackage{enumerate}	
\usepackage{paralist}
\usepackage{cite}
\usepackage{algorithm}
\usepackage{algorithmicx}
\usepackage{algpseudocode}
\usepackage{authblk}%作者的单位出现在名字下方，且不影响脚注
\usepackage{mathtools}%让文中公式命令flalign,equation可以实现
\usepackage{pifont}%让公式可以用带圈的数字进行标号
\usepackage[numbers,sort&compress]{natbib}%使引用格式[1,2,3,4,5]变成[1-5]
\usepackage{comment}%隐去参考文献的宏包
\usepackage{tikz}
\usetikzlibrary{positioning}

\hypersetup
{
	colorlinks=true,
	linkcolor=blue,
	filecolor=blue,
	urlcolor=blue,
	citecolor=cyan,
}

\usepackage{etoolbox}

\newcommand{\cl}{{\rm cl}}

\newcommand{\co}{{\rm co}}
\newcommand{\Lp}{L^{-1}}
\newcommand{\mH}{\mathcal{H}}
\newtheorem{theorem}{Theorem}[section]

\newtheorem{conjecture}[theorem]{Conjecture}

\theoremstyle{definition}
\newtheorem{definition}[theorem]{Definition}

\AtBeginEnvironment{proof}{\setcounter{case}{0}}
{\setlength{\leftmargini}{1.5\parindent}
	\begin{itemize}
		\setlength{\itemsep}{-1.1mm}}
	{\end{itemize}}

\begin{document}
%\pagewiselinenumbers
\title{\bf Hamiltonian Properties of 3-Connected Claw-Free Graphs and Line Graphs of 3-Hypergraphs}
\author[1]{\bf Kenta Ozeki \footnote{Email: ozeki-kenta-xr@ynu.ac.jp.}}
\author[1,2]{\bf Leilei Zhang \footnote{Email: zhang-leilei-kn@ynu.ac.jp. (corresponding author)}}
	
\affil[1]{\footnotesize Faculty of Environment and Information Sciences, Yokohama National University, Yokohama 240-8501, Japan}
\affil[2]{\footnotesize School of Mathematics and Statistics, and Key Laboratory of Nonlinear Analysis and Applications (Ministry of Education), Central China Normal University, Wuhan 430079, China}
\date{}%{\today}
\maketitle
\begin{abstract}
Motivated by  Thomassen's well-known line graph conjecture, many researchers have explored sufficient conditions for claw-free graphs to be Hamiltonian or Hamilton-connected. In 1994, Ageev proved that every $2$-connected claw-free graph with domination number at most $2$ is Hamiltonian. In this paper, we extend this line of research to $3$-connected graphs by establishing the best possible upper bound on the domination number that guarantees Hamiltonicity. Specifically, we show that, except for some well-defined exceptional graphs, every $3$-connected claw-free graph $G$ with domination number 
%$\gamma(G) \leq 5$ 
at most $5$ is Hamiltonian. Furthermore, we prove that, apart from a few exceptional cases, every $3$-connected claw-free graph $G$ with domination number 
%$\gamma(G) \leq 4$ 
at most $4$ is Hamilton-connected, thereby generalizing earlier results of Zheng, Broersma, Wang and Zhang 
%$(\gamma(G)=1)$ 
and Vr\'ana, Zhan and Zhang.
%$(\gamma(G)\le 3)$. 
%In 2020, Li, Ozeki, Ryjáček, and Vrána conjectured that every 4-connected line graph of a 3-hypergraph is Hamilton-connected, and showed that this conjecture is equivalent to the  line graph conjecture of Thomassen. Motivated by this, 
We further investigate the Hamiltonian properties of line graphs of $3$-hypergraphs, and prove that every 3-connected line graph of a 3-hypergraph with domination number at most $4$ is Hamiltonian.

		
\end{abstract}
	
%---------------------
\section{Introduction}
We generally follow the most common graph-theoretical notation and terminology, and for notations and concepts not defined here we refer the reader to Bondy and Murty \cite{BM}. Specifically, by a graph we always mean a simple finite undirected graph; whenever we admit multiple edges, we always speak about a multigraph. A {\it Hamilton cycle} (resp. {\it Hamilton path}) in a graph is a cycle (resp. path) passing through all vertices of the graph. If a graph has a Hamilton cycle, then we say the graph is Hamiltonian. A graph is called {\it Hamilton-connected} if between any two distinct vertices there is a Hamilton  path. The complete bipartite graph on $s$ and $t$ vertices is denoted by $K_{s,\, t}$. The graph $K_{1,3}$ is called the {\it claw}. A graph is called {\it claw-free} if it contains no induced subgraph isomorphic to the claw. A subset $X$ of vertices in a graph $G$ is called a {\it dominating set} if every vertex of $G$ is either contained in $X$ or adjacent to some vertex of $X$. The {\it domination number} of a graph $G$, denoted by $\gamma(G)$, is the size of a smallest dominating set of $G$. A subset $D \subseteq E(G)$ of edges in a graph $G$ is called an {\it  edge dominating set}  if every edge in $G$ is either in $D$ or adjacent (i.e., shares a common vertex) to some edge in $D$. The {\it edge domination number} of $G$ is the size of the smallest edge dominating set in $G$. For a positive integer \(k\), the notation \([k]\) represents the set of integers from \(1\) to \(k\). 

The line graph of a graph $H$, denoted by $L(H)$, is the graph $G$ with $V(G)=E(H)$, in which two vertices are adjacent if and only if the corresponding edges of $H$ have at least one vertex in common.  In 1986, Thomassen \cite{T} posed one of the most fascinating conjectures related to the Hamiltonicity of line graphs, as stated in Conjecture \ref{conj1} (i).  Note that every line graph is claw-free. A (seemingly) stronger version (Conjecture \ref{conj1} (ii)) was established by Matthews and Sumner \cite{MS}. Ryj\'{a}\v{c}ek~\cite{R} later showed that these two conjectures are in fact equivalent. Ku\v{z}el and Xiong~\cite{KX} strengthened Conjecture~\ref{conj1} (i), and Ryj\'{a}\v{c}ek and Vr\'{a}na~\cite{RP} strengthened Conjecture~\ref{conj1} (ii), both to Hamilton-connectedness. Finally, Ryj\'{a}\v{c}ek and Vr\'{a}na~\cite{RP} proved that all four statements in Conjecture~\ref{conj1} are mutually equivalent.

\begin{conjecture}\label{conj1}
{\rm(i) (Thomassen \cite{T})} Every $4$-connected line graph is Hamiltonian.
\newline\indent{\rm(ii) (Matthews and Sumner \cite{MS})} Every $4$-connected claw-free graph is Hamiltonian.
\newline\indent{\rm(iii) (Ku\v{z}el and Xiong \cite{KX})} Every $4$-connected line graph is Hamilton-connected.
\newline\indent{\rm(iv) (Ryj\'{a}\v{c}ek and Vr\'{a}na \cite{RP})} Every $4$-connected claw-free graph is Hamilton-connected.
\end{conjecture}

Researchers have approached the above conjectures from many different angles, including constraints on graph connectivity \cite{KV}, diameter \cite{MLLZ}, forbidden subgraphs \cite{LRVXY,LRVXY2}, and minimum degree conditions \cite{CLX}. In particular, in 1994, Ageev \cite{A} established the following sufficient condition for Hamiltonicity in claw-free graphs based on the domination number.

\begin{theorem}\label{Agvee}{\rm (Agvee \cite{A})}
Every $2$-connected claw-free graph with domination number at most $2$ is Hamiltonian.
\end{theorem}

The conditions in this theorem (2-connected, claw-free, and the domination number bound) are all best possible, as shown in \cite{ZBWZ}. Motivated by this theorem, Cai \cite{C} and Zheng, Broersma, Wang and Zhang \cite{ZBWZ} relaxed the claw-free condition by imposing specific structural constraints on claws whenever they appear as induced subgraphs.  As a potential approach to Conjecture \ref{conj1}, several researchers have explored properties stronger than being Hamiltonian. For example, Zheng, Broersma, Wang and Zhang \cite{ZBWZ}  proved that every 3-connected claw-free graph with domination number 1 is Hamilton-connected. Building on this, they conjectured that the same holds for graphs with domination number 2. Recently, this conjecture was confirmed by Vr\'{a}na, Zhan and Zhang  \cite{VZZ}. In fact, they proved a stronger version of the conjecture, stated in the following theorem.

\begin{theorem}\label{VZZ1} {\rm  (Vr\'{a}na, Zhan and Zhang \cite{VZZ})}
	Every $3$-connected claw-free graph with domination number at most $3$ is Hamilton-connected.
\end{theorem}

As an immediate consequence, we have the following result:

\begin{theorem}\label{VZZ2}{\rm  (Vr\'{a}na, Zhan and Zhang \cite{VZZ})}
	Every $3$-connected claw-free graph with domination number at most $3$ is Hamiltonian.
\end{theorem}

In this paper, 
%we further investigate the Hamiltonicity of $3$-connected claw-free graphs by studying the relationship between their domination number and Hamiltonicity. We establish the best possible upper bound on the domination number that guarantees Hamiltonicity in this class of graphs.
we extend Theorems \ref{VZZ1} and \ref{VZZ2} in several directions;
characterizing the extremal graphs
in Theorem \ref{VZZ2} (see Theorem \ref{thm1})
and those in Theorem \ref{VZZ1} (see Theorem \ref{thm2}),
and also extending Theorem \ref{VZZ2} to 
$3$-connected line graph of a $3$-hypergraph (see Theorem \ref{thm3}).

For our results, we will need some more terminology. We denote the neighborhood of a vertex $x \in V(G)$ in a graph $G$ by $N_G(x)$, which is the set of all vertices adjacent to $x$. If $X \subseteq V(G)$, then $N_G(X)$ is defined as the set of all vertices in $V(G) \setminus X$ that have a neighbor in $X$.  We say that a vertex $x\in V(G)$ is {\it eligible} if $G[N_G(x)]$ is a connected noncomplete graph. The closure $\cl(G)$ of a claw-free graph $G$ was introduced by Ryj\'{a}\v{c}ek \cite{R}, who defined it as the graph obtained by recursively applying local completion at eligible vertices. In the same paper, Ryj\'{a}\v{c}ek also proved that the closure $\cl(G)$ is uniquely determined and that $G$ is Hamiltonian if and only if $\cl(G)$ is Hamiltonian. This result established closure as a powerful tool in the study of Hamiltonicity of claw-free graphs.

We also recall two well-known graphs that will occur as exceptions in our results, namely, the Petersen graph $P$ and the Wagner graph $W$ (see Fig \ref{fig1}). It is well known that the Wagner graph can be obtained from the Petersen graph by deleting an arbitrary edge and then suppressing the two resulting vertices of degree $2$—that is, by replacing  vertex of degree $2$ and its incident edges with a single edge connecting its neighbors. An edge is said to be {\em pendant} if one of its endvertices has degree $1$. 

\begin{figure}[!ht]
	\centering
	\includegraphics[width=90mm]{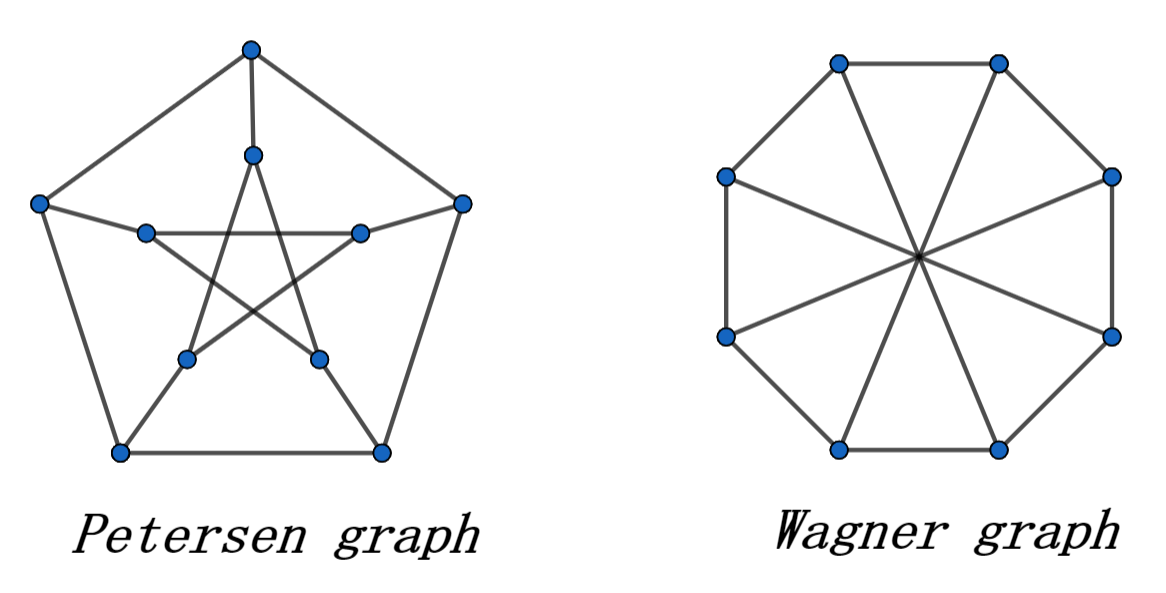}
	\caption{The Petersen graph $P$ and the Wagner graph $W$.}
	\label{fig1}
\end{figure}

Let $\mathcal{P}'$ denote the class of graphs whose edge domination number is at most $5$
and that are obtained from the Petersen graph by modifying each vertex in one of the following two ways: (1) attaching at least one pendant edge to the vertex, or (2) subdividing some edges incident to the vertex.
%such that the resulting graph has edge domination number at most $5$. 
The first main contribution of this paper is the following theorem, which extends Theorem~\ref{VZZ2} and gives the best possible upper bound on the domination number ensuring Hamiltonicity in $3$-connected claw-free graphs.

\begin{theorem}\label{thm1}
Let $G$ be a $3$-connected claw-free graph of order $n$. If the domination number of $G$ is at most $5$, then $G$ is Hamiltonian unless its closure $\mathrm{cl}(G)$ is the line graph of a graph belonging to the class $\mathcal{P}'$.
\end{theorem}

Motivated by Theorem~\ref{VZZ1} and the work of Zheng, Broersma, Wang, and Zhang~\cite{ZBWZ}, we further investigate the relationship between the domination number and Hamilton-connectedness in claw-free graphs. In this paper, we prove that, with the exception of certain explicitly characterized graphs, every $3$-connected claw-free graph with domination number at most $4$ is Hamilton-connected. This result extends Theorem~\ref{VZZ1} and determines the best possible upper bound on the domination number that guarantees Hamilton-connectedness in $3$-connected claw-free graphs. 

We need some notation to introduce our next theorem.
Let $\mathcal{W}'$ denote the class of graphs 
whose edge domination number is at most $4$
and that are obtained from the Wagner graph by performing, at each vertex $v$, one or more of the following operations: attaching at least one pendant edge or a double edge (i.e., an edge of multiplicity 2); subdividing some edges incident to $v$; and optionally adding a double edge between $v$ and any subdivision vertex. 
%The resulting graph is required to have edge domination number at most $4$.
As noted in \cite{BFR}, the closure operation does not preserve the (non-)Hamilton-connectedness of $G$. To address this problem, Ryj\'{a}\v{c}ek and Vr\'ana \cite{RP} introduced a strengthened version of the closure operation, known as the \textit{$\textbf{M}$-closure}, denote by $\cl^M(G)$.  In this paper, we do not require the explicit construction of $\cl^M(G)$; instead, we only rely on some of its properties, which will be listed later.

\begin{theorem}\label{thm2}
 Let $G$ be a $3$-connected claw-free graph with order $n$. If the domination number of $G$ is at most $4$, then $G$ is Hamilton-connected, unless its $\textbf{M}$-closure is the line graph of a graph belonging to the class $\mathcal{W}'$.
\end{theorem}

A {\it hypergraph} $\mH$ is an ordered pair $(V(\mH),E(\mH))$, where $V(\mH)$ is the vertex set of $\mH$ and $E(\mH)$ is a collection of not necessarily distinct nonempty subsets of $V(\mH)$, called {\it hyperedges} of $\mH$. The {\it rank} of a hypergraph $\mH$ is the maximum cardinality of a hyperedge of $\mH$. A hypergraph of rank $r$ is also referred to as an {\it $r$-hypergraph}. Note that a $2$-hypergraph corresponds to a  multigraph.  A hyperedge of cardinality 2 will be sometimes also called an edge of $\mH$. Thus, a hypergraph without hyperedges is a multigraph, and a multigraph without parallel edges and without loops is a graph.

In 2020, Li, Ozeki, Ryj\'{a}\v{c}ek and Vr\'{a}na \cite{LORV} found a close relation between line graphs of 3-hypergraphs and Conjecture \ref{conj1}.

%\begin{conjecture}{\rm (Li, Ozeki, Ryj\'{a}\v{c}ek and Vr\'{a}na \cite{LORV})}
%Every $4$-connected line graph of a $3$-hypergraph is Hamilton-connected.
%\end{conjecture}

\begin{theorem}\label{thm3}
{\rm (Li, Ozeki, Ryj\'{a}\v{c}ek and Vr\'{a}na \cite{LORV})}
Conjecture~\ref{conj1} is equivalent to the statement that
every $4$-connected line graph of a $3$-hypergraph is Hamilton-connected.
\end{theorem}

%They showed that this conjecture is equivalent to Thomassen's well-known line graph conjecture (Conjecture~\ref{conj1}). 

Motivated by this theorem, we further investigate the domination number of the line graph of a 3-hypergraph, and we establish the following result

\begin{theorem}\label{thm3}
	Every $3$-connected line graph of a  $3$-hypergraph with domination number at most $4$ is Hamiltonian.
\end{theorem}

Since every graph can be viewed as a 3-hypergraph, Theorem \ref{thm3} can be regarded as a natural generalization of both Theorem \ref{Agvee}.
% and Theorem \ref{thm1}. 

We organize the paper as follows. In Section \ref{proof12_sec}, we present necessary preliminaries and provide the proofs of Theorems~\ref{thm1} and~\ref{thm2}. In Section \ref{proof3_sec}, we present the proof of Theorem~\ref{thm3} and also discuss several related problems that arise naturally from Theorem~\ref{thm3}.

\section{Proofs of Theorems \ref{thm1} and \ref{thm2}}
\label{proof12_sec}

\subsection{Preliminary}

We first recall some necessary known concepts and results. We say that a vertex is {\it simplicial} if its neighbors induce a complete graph. The following Theorem was proved in \cite{RP}.

\begin{theorem} {\rm (Ryj\'a\v{c}ek and Vr\'ana \cite{RP})}\label{thm4}
Let $G$ be a connected line graph of a multigraph. Then there is, up to an isomorphism, a uniquely determined multigraph $H$ with $G=L(H)$ such that a vertex $e\in V(G)$ is simplicial in $G$ if and only if the corresponding edge $e\in E(H)$ is a pendant edge in $H$.
\end{theorem}

The multigraph $H$ with the properties given in Theorem \ref{thm4} will be called the {\it preimage} of a line graph $G$ and denoted by $H=\Lp(G)$. We will also use the notation $u=L(e)$ and $e=\Lp(u)$ for an edge $e\in E(H)$ and the corresponding vertex $u\in V(G)$. An edge-cut $R\subset E(H)$ of a multigraph $H$ is {\em essential} if $H-R$ has at least two nontrivial components. Let $H$ be a graph with at least $k+1$ edges. We say that $H$ is {\em essentially $k$-edge-connected}  if every essential edge-cut in $H$ has size at least $k$. It is a well-known fact that a line graph $G$ is $k$-connected if and only if $\Lp(G)$ is essentially $k$-edge-connected (see \cite[Proposition 1.1.3]{S}). The following properties of the closure will be used in our proofs.

\begin{theorem}\label{thm5} {\rm (Ryj\'a\v{c}ek \cite{R})}  Let $G$ be a claw-free graph and $\cl(G)$ its closure. Then
\newline\indent {\rm (i) }there is a triangle-free graph $H$ such that $\cl(G)$ is the line graph of $H$, and 
\newline\indent {\rm (ii)} $G$ is Hamiltonian if and only if its closure $\cl(G)$ is Hamiltonian.
\end{theorem}

Note that in \cite{BFR}, Brandt, Favaron, and Ryj\'{a}\v{c}ek observed that the closure operation does not preserve the (non-)Hamilton-connectedness of a graph~$G$. To overcome this problem, the concept of the  {\it  \textbf{M}-closure} $\cl^M(G)$ for a claw-free graph $G$ was introduced in \cite{RP}. We omit the technical details of its construction, as they are not required for our proofs. The following result summarizes the key properties of $\cl^M(G)$ that we will use in our arguments:

\begin{theorem}{\rm (Ryj\'{a}\v{c}ek and Vr\'{a}na \cite{RP})\label{thm9}}
	Let $G$ be a claw-free graph. Then $\cl^M(G)$ has the following properties:
	\newline\indent {\rm (i)} $V(G)=V(\cl^M(G))$ and $E(G)\subset E(\cl^M(G))$;
	\newline\indent {\rm (ii)} $\cl^M(G)$ is uniquely determined;
	\newline\indent {\rm (iii)} $G$ is Hamilton-connected if and only if $\cl^M(G)$ is Hamilton-connected,
	\newline\indent {\rm (iv)} $\cl^M(G)=L(H)$, for some multigraph $H$ .
\end{theorem}

For $x, y \in V(G)$, a path (resp.~trail) with endvertices $x, y$ is referred to as an $(x, y)$-path (resp. $(x, y)$-trail), a trail with terminal edges
$e, f \in E(G)$ is called an $(e, f)$-trail. A set of vertices $S\subset V(G)$ {\em dominates} an edge $e$ if $e$ has at least one vertex in $S$, and a subgraph $F \subseteq    G$ dominates $e$ if $V(F )$ dominates $e$. A closed trail $T$ is a {\em dominating closed trail} if $T$ dominates all edges of $G$, and an $(e_1,e_2)$-trail is an {\em internally dominating $(e_1,e_2)$-trail } (abbreviated $(e_1,e_2)$-IDT) if the set of its interior vertices dominates all edges of $G$.

Harary and Nash-Williams \cite{HN} established a fundamental correspondence between a dominating closed trail in a graph $H$ and a Hamilton cycle in its line graph $G=L(H)$. Extending this result to Hamilton-connectedness, Li, Lai, and Zhan \cite{LLZ} proved that $G=L(H)$ is Hamilton-connected if and only if $H$ contains an $(e_1, e_2)$-IDT for every pair of edges $e_1, e_2 \in E(H)$.

\begin{theorem}\label{thm6}
Let $H$ be a multigraph with $|E(H)|\geq 3$ and let $G=L(H)$.
\newline\indent{\rm (i) (Harary and Nash-Williams \cite{HN})} The graph $G$ is Hamiltonian if and only if $H$ has a 
\newline\indent\quad\,\,\,\,dominating closed trail.
\newline\indent{\rm (ii) (Li, Lai and Zhan \cite{LLZ})} For every $e_i\in E(H)$ and $a_i=L(e_i)$, $i=1,2$, $G$ has a Hamilton   \newline\indent\quad\,\,\,\,$(a_1,a_2)$-path if and only if $H$ has an $(e_1,e_2)$-IDT.
\end{theorem}

The concept of the core of a graph plays an important role in studying the Hamiltonian properties of line graphs. The  {\em  core}  of a graph $H$, denoted by ${\rm co}(H)$, is the multigraph obtained by iteratively suppressing all vertices of degree 2 (that is, replacing each such vertex and its incident edges with a single edge between its neighbors), and by removing all pendant edges. Shao \cite[Lemma 1.4.1]{S} established the following properties of the core of a multigraph.

\begin{theorem}{\rm (Shao \cite{S})}\label{thm7}
Let $H$ be an essentially $3$-edge-connected multigraph. Then
\newline\indent {\rm (i)} $\co(H)$ is uniquely determined;
\newline\indent {\rm (ii)} $\co(H)$ is $3$-edge-connected.
\end{theorem}

Let $H$ be a multigraph, $R\subset H$ a spanning subgraph of $H$, and let $\mathcal{R}$ be the set of components of $R$. Then $H/R$ is the multigraph with $V(H/R)=\mathcal{R}$, in which, for each edge in $E(H)$ between two components of $R$, there is an edge in $E(H/R)$ joining the corresponding vertices of $H/R$. The contraction operation maps the vertex set $V(H)$ onto $V(H/R)$, where each component of the subgraph $R \subseteq H$ is contracted to a single vertex in $H/R$. The multigraph $H/R$ is said to be the contraction of $H$ by  $R$. If $H/R \cong F$, then this contraction defines a mapping $\alpha: H \rightarrow F$, which is called the {\em contraction} of $H$ to $F$. For a component $R_1$ of $R$ and a vertex $u \in V(F)$ such that $\alpha(R_1) = u$, we refer to $R_1$ as the  {\em preimage} of the vertex $u $.

The celebrated Nine-Point Theorem by Holton, McKay, Plummer and Thomassen \cite{HMPT} asserts that every 3-connected cubic graph contains a cycle passing through any set of nine vertices. This result was later strengthened by Bau and Holton \cite{BH}, who proved that every 3-connected cubic graph contains a cycle through any twelve prescribed vertices, unless the graph can be contracted to the Petersen graph.  Later, Chen, Lai, Li, Li, and Mao \cite{CLLLM} extended this result to 3-connected graphs. Their generalization is formulated as follows:

\begin{theorem}{\rm  (Chen, Lai, Li, Li and Mao \cite{CLLLM})}\label{thm8}
 Let $H$ be a $3$-edge-connected multigraph and let $A\subset V(H)$ with $|A|\leq12$. Then either $H$ has a closed trail $T$ such that $A\subset V(T)$, or $H$ can be contracted to the Petersen graph in such a way that the preimage of each vertex of the Petersen graph contains at least one vertex in $A$.
\end{theorem}

To prove Theorem 1.6, we utilize a result from \cite{LRVXY}, which is based on a stronger theorem by Bau and Holton \cite{BH-2} concerning cycles in $3$-connected cubic graphs that pass through a given set of vertices and a specified edge.

\begin{theorem}{\rm (Liu, Ryj\'{a}\v{c}ek, Vr\'{a}na, Xiong and Yang \cite{LRVXY})},\label{thm8-1}
	Let $H$ be a $3$-edge-connected multigraph, $A\subset V(H)$ with $|A|=8$, and let $e\in E(H)$. Then either $H$ contains a closed trail $T$ such that $A\subset V(T)$ and $e\in E(T)$, or there is a contraction $\alpha:H\rightarrow P$ such that  $\alpha(e)=xy\in E(P)$ and $\alpha(A)=V(P)\setminus\{x,y\}$.
\end{theorem}

In order to apply Theorems~\ref{thm8} and~\ref{thm8-1}, we first reduce the problem to the core of $H$. Note that if $H$ is essentially $3$-edge-connected, then no two vertices of degree less than $2$ are adjacent. Hence, each edge of $H$ can contain at most one vertex of degree less than $2$.   Let $e_1$ and $e_2$ be two edges of $H$, and let $e_1', e_2'$ denote their corresponding edges in $\mathrm{co}(H)$. The definition of each $e_i'$ $(i = 1, 2)$ depends on how the edge $e_i$ is treated during the construction of $\mathrm{co}(H)$ (see Fig. 2), and is given as follows:

\begin{figure}[!ht]
	\centering
	\includegraphics[width=110mm]{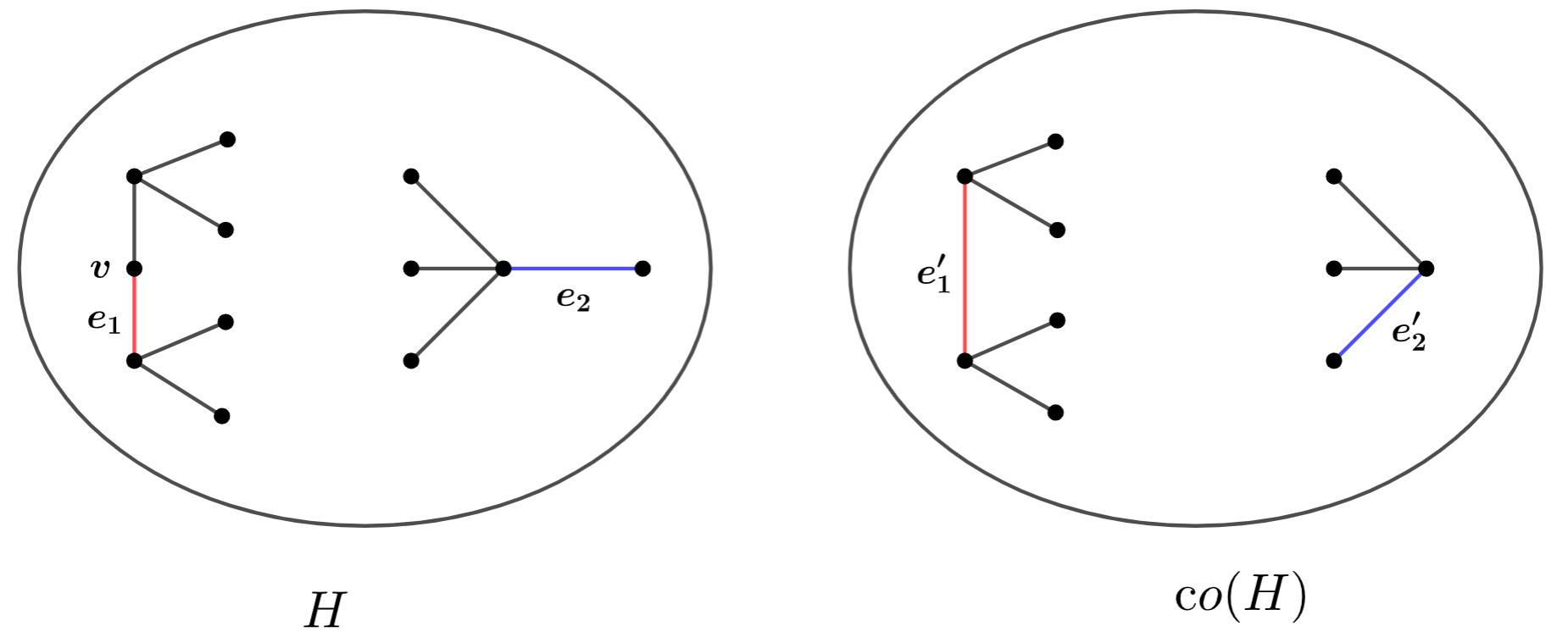} 
	\caption{Illustrative examples for the construction in Definition \ref{Defi1}.}
	\label{fig2}
\end{figure}

\begin{definition}\label{Defi1}
		Let $e_1$ and $e_2$ be two edges that were removed during the construction of $\mathrm{co}(H)$. The corresponding edges $e_1', e_2' \in E(\mathrm{co}(H))$ are defined by one of the following rules, depending on the role of $e_1$ and $e_2$ in the reduction process:
	
	\begin{itemize}
		\item[(i)] If one of the endvertices of $e_1$, say $v$, has degree two, then $e_1'$ is defined as the edge in $\mathrm{co}(H)$ that results from suppressing $v$.
		
		\item[(ii)] If $e_2$ is a pendant edge, then $e_2'$ is defined as an arbitrary edge in $\mathrm{co}(H)$ incident to the non-leaf endvertex of $e$.
	\end{itemize}
\end{definition}

\subsection{Proof of Theorem \ref{thm1}}

Now we are ready to give a proof of Theorem \ref{thm1}. It should be noted that certain parts of the proof of Theorem~\ref{thm1} coincide with those in the proof of Theorem 1 in \cite{VZZ}. However, since certain technical details require further discussion, and for the sake of completeness, we provide a full proof here (including the identical parts).

\bigskip

\noindent{\bf Proof of Theorem \ref{thm1}.}  Let $G$ be a 3-connected claw-free graph with domination number at most 5. Suppose, for the sake of contradiction, that $G$ is not Hamiltonian. As $V(\cl(G))=V(G)$, $\gamma(\cl(G))\le \gamma(G)$ and $\cl(G)$ is 3-connected, the graph $\cl(G)$ satisfies the same hypotheses as $G.$ By Theorem \ref{thm5}, the graph $G$ is Hamiltonian if and only if its closure $\cl(G)$ is Hamiltonian. Hence it suffices to prove Theorem \ref{thm1} for closed claw-free graphs. Let $H=\Lp(G).$ Since $G$ is 3-connected, $H$ is essentially $3$-edge-connected. By Theorem~\ref{thm6}, $H$ has no dominating closed trail, and by Theorem~\ref{thm7}, the core $\co(H)$ is $3$-edge-connected.

In the case that the edge $e$ is removed in the construction of $\mathrm{co}(H)$, we define the corresponding edge $e'$ in $\mathrm{co}(H)$ as specified in Definition~\ref{Defi1}. For convenience, if the edge $e$ remains in $\mathrm{co}(H)$ (i.e., both of its endvertices have degree at least $3$), we simply let $e' = e$. See the example in Fig. \ref{fig2}. One can easily check that if an edge is dominated by $e$, it must also be dominated by $e'$.

Let $\{u_1,u_2,u_3,u_4,u_5\}$ be a dominating set of $G$. By the definition of a line graph, the set of the five corresponding edges in $H$, denoted by $\{e_1,e_2,e_3,e_4,e_5\}$ dominate all edges of $H$.  For each $e_i \in \{e_1, e_2, e_3,e_4,e_5\}$, we identify an edge $e_i'$ in $\mathrm{co}(H)$ using exactly the same rule as described for $e'$ in Definition~\ref{Defi1}. Let $e'_i = z_{2i-1}z_{2i}$ for each $i \in [5]$, that is, the endvertices of $e'_i$ are $z_{2i-1}$ and $z_{2i}$.  Since the set $\{z_1, \ldots, z_{10}\}$ dominates all edges of $\mathrm{co}(H)$, the existence of a closed trail in $\mathrm{co}(H)$ that passes through all these vertices implies the existence of a dominating closed trail in $\mathrm{co}(H)$. It is straightforward to verify that, for every possible construction of $e'_i$, such a dominating closed trail in $\mathrm{co}(H)$ corresponds to a dominating closed trail in $H$, which contradicts our earlier assumption. Then, by Theorem~\ref{thm8}, $\co(H)$ can be contracted to the Petersen graph such that the preimage of each vertex in the Petersen graph contains a vertex from $\{z_1, \ldots, z_{10}\}$. Let $p_i , i\in [10]$ be the ten vertices of the Petersen graph, and $H_i$ be the preimage of $p_i, i\in [10]$ such that $z_i\in V(H_i)$. According to the definition of the contraction, the neighborhood of $H_i$ in $\mathrm{co}(H)$ satisfies $|N_{\mathrm{co}(H)}(H_i)| = 3$.

We now prove that $H_i$ consists of a single vertex. Since ${z_1, \ldots, z_{10}}$ dominates all edges of $\co(H)$ and only $z_i$ lies in $H_i$, every edge of $H_i$ must be dominated by $z_i$.
Let $x_1, \dots , x_j$ be 
%$x_j$ $(j \ge 0)$ be 
the vertices of $H_i$ adjacent to $z_i$, where $j \geq 0$.
If $j \ge 2$, then because $|N_{\co(H)}(H_i)| = 3$, at least one of $x_1$ or $x_2$ has degree at most $2$, contradicting the assumption that $\co(H)$ has minimum degree at least $3$. If $j = 1$, then either $x_1$ or $z_i$ must have degree at most $2$, again a contradiction. Therefore, the only possibility is $j = 0$, and hence $H_i$ consists of the single vertex $z_i$.

Thus, we have $\co(H) = P$. By the definition of $\co(H)$ and the fact that $H$ is essentially 3-edge-connected, it follows that $H$ is obtained from $\co(H)$ by adding pendant edges to its vertices and by subdividing some of its edges. Note that if there exists a vertex in $P$ to which no pendant edge is attached and none of its incident edges are subdivided, then one can find a dominating closed trail in $H$, contradicting Theorem~\ref{thm6}.

 This completes the proof.\hfill $\Box$

\subsection{Proof of Theorem \ref{thm2}}

\begin{figure}[!ht]
\centering
  \includegraphics[width=90mm]{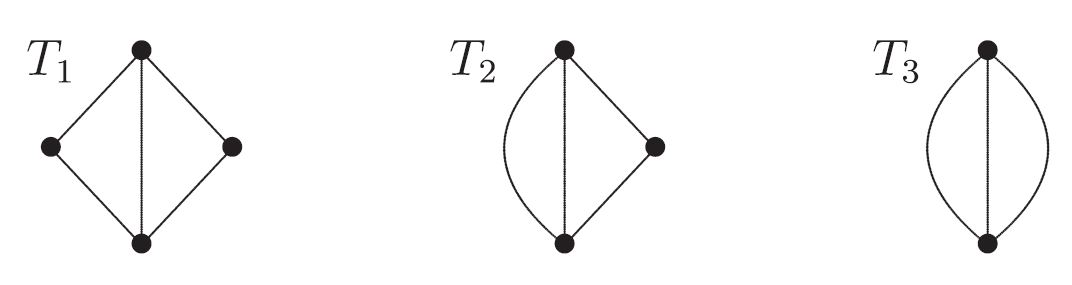}
  \caption{The diamond $T_1$, the multitriangle $T_2$ and the triple edge $T_3$.}
  \label{fig3}
\end{figure}

\begin{theorem}{\rm (Ryj\'{a}\v{c}ek and Vr\'{a}na \cite{RP})\label{thm10}}
Let $G$ be a claw-free graph and let $T_1$, $T_2$, $T_3$ be the multigraphs shown in Figure \ref{fig3}. Then $G$ is $\textbf{M}$-closed if and only if $G$ is a line graph of a multigraph and $\Lp(G)$ does not contain a subgraph (not necessarily induced) isomorphic to any of the multigraphs $T_1$, $T_2$ or $T_3$.
\end{theorem}

Now we are ready to give a proof of  Theorem \ref{thm2}. Since the proof follows a similar line of reasoning as that of Theorem \ref{thm1}, we omit some of the details.
\medskip

\noindent{\bf Proof of Theorem \ref{thm2}.}  Let $G$ be a 3-connected claw-free graph with domination number at most 4 that is not Hamilton-connected. Let $H=\Lp(\cl^M(G))$. Since $G$ is 3-connected, $H$ is essentially $3$-edge-connected. By Theorem \ref{thm6}, there exist two edges $e_1$ and $e_2$ such that the graph $H$ does not contain an $(e_1, e_2)$-IDT. To derive a contradiction, we first reduce the problem to the core of $H$. Let $e_1'$ and $e_2'$ be defined as in Definition~\ref{Defi1}.

Now, in the multigraph $\co(H)$, if $e'_1=e'_2$, we set $\widetilde{e}=e'_1=e'_2$; otherwise we subdivide the edges $e'_1,e'_2$ and join the two new vertices with an edge $\widetilde{e}$. In both cases, we denote by $\co(H)'$ the resulting multigraph. In the first case, $\co(H)'=\co(H)$ and it is $3$-edge-connected by Theorem~\ref{thm7}, and it is easy to see that in the second case, $\co(H)'$ is also $3$-edge-connected (recall that, in $\co(H)$, all pendant edges are removed and all vertices of degree 2 are suppressed).

Let $\{v_1,v_2,v_3,v_4\}$ be a dominating set of $G$. Then, by the definition of a line graph, the set of the four corresponding edges, denoted by $\{f_1, f_2, f_3, f_4\}$, dominates all edges of $H$. For $f_i\in E(H), i=1,2,3,4 $, we find edges $f'_i$ in $\co(H)$. For each $i \in [4]$, let $f'_i = w_{2i-1}w_{2i}$; that is, the endvertices of $f'_i$ are $w_{2i-1}$ and $w_{2i}$. Let $A = \{w_1, \ldots, w_8\}$. By Theorem~\ref{thm8-1}, either there exists a closed trail in $\co(H)'$ containing the vertices $w_1, \ldots, w_8$ and the edge $\widetilde{e}$, or there exists a contraction $\alpha': \co(H)' \to P$ such that $\alpha'(\widetilde{e}) = xy \in E(P)$ and $\alpha'(A) = V(P) \setminus \{x, y\}$. In the first case, we can find an $(e_1,e_2)$-IDT in $H$ (recall that the corresponding set of vertices $\lbrace w_1,\ldots,w_8\rbrace$ dominates all edges of $H$), a contradiction. So we can deduce that there is a contraction $\alpha':\co(H)'\rightarrow P$ such that  $\alpha'(\widetilde{e})=xy\in E(P)$ and $\alpha'(A)=V(P)\setminus\{x,y\}$. Let $W_i,\,i\in \{1,\ldots,8\}$ be the preimage in $\co(H)'$ that contains $w_i$ and let $W_9,W_{10}$ be the preimage of the two endvertices of $\widetilde{e}$. Since $\{w_1, \ldots, w_8\}$ dominates all edges of $\mathrm{co}(H)'$, it follows that both $W_9$ and $W_{10}$ consist of a single vertex, and that $e'_1 \neq e'_2$. Recall that the Wagner graph can be obtained from the Petersen graph by deleting an arbitrary edge and suppressing the two resulting vertices of degree $2$. Hence, the contraction $\alpha'$ can contract $\mathrm{co}(H)$ to the Wagner graph $W$.

Recall that the set ${w_1, \ldots, w_8}$ is a dominating set of $\co(H)$, and each $w_i$ belongs to the corresponding component $W_i$, so every edge of $W_i$ is dominated by $w_i$. Thus, each $W_i$ is a star (possibly with multiple edges). By Theorem~\ref{thm10}, we conclude that each $W_i$ contains no triple edges (Fig \ref{fig3}, $T_3$). Note that $|N_{\co(H)}(W_i)| = 3$ and that $w_i$ is an end-vertex of $f'_i$. Combined with the fact that the set $\{f'_1, f'_2, f'_3, f'_4\}$ dominates all edges of $\co(H)$, we conclude that $w_i$ is adjacent to vertices in other preimages. It follows that, after removing all pendant edges and pendant double edges, the number of double edges in $W_i$ is at most two. Based on the structural properties of $\co(H)$ and the fact that its domination number is 4, we conclude that $H$ belongs to the class $\mathcal{W}'$.

This completes the proof.\hfill $\Box$

\section{Proof of Theorem \ref{thm3}}
\label{proof3_sec}

\subsection{Preliminaly}

If $\mathcal{H}$ is a 3-hypergraph, then the {\it incidence graph }$IG(\mathcal{H})$ is the graph obtained from $\mathcal{H}$ by subdividing every edge with a new vertex of degree 2 and replacing every 3-hyperedge by a new vertex, adjacent to all its three vertices. The newly added vertices
will be referred to as white vertices and the original vertices of $\mathcal{H}$ as black vertices of $IG(\mathcal{H})$. Thus, in $IG(\mathcal{H})$, white vertices correspond to the edges and hyperedges of $\mathcal{H}$, and black vertices correspond to the vertices of $\mathcal{H}$. Note that every white vertex in $IG(\mathcal{H})$ has degree 2 or 3.

A {\it closed walk} $Q$ in a graph $G$ is a sequence
$$
v_0, e_0, v_1, e_1, \ldots, e_{k-1}, v_k,
$$
such that for each $0 \leq i \leq k - 1$, the edge $e_i$ connects the vertices $v_i$ and $v_{i+1}$, and $v_k = v_0$. Each vertex $v_i$ is said to be visited by $Q$, with {\it multiplicity} equal to the number of times it appears in the sequence. Similarly, an edge $e_i$ is said to be visited by $Q$ (again with possible multiplicity). A closed trail is a closed walk visiting  each edge at most once.

Let $v_i$ be a vertex visited once by the above walk. The predecessor edge of $v_i$ is defined to be $e_{i-1}$ (with subtraction  modulo $k$). Similarly, the successor edge of $v_i$ is $e_i$ if $i < k$ and $e_0$ otherwise. Given an arbitrary set $W$ of vertices of degree $2$ or $3$ in $G$, a closed  $W$-quasitrail in $G$ is a closed walk which traverses  each edge at most twice, and if an edge $e$ is traversed twice, then it has an endvertex $w\in W$ such that $w$ is visited once  and $e$ is both the predecessor edge and the successor edge of $w$. A closed $W$-quasitrail in $G$ is dominating if it visits at least  one vertex in every edge of $G$.

We will use the following characterization of $3$-hypergraphs with Hamiltonian line graphs, which follows from  Corollary 7 in \cite{LORV}.

\begin{theorem}{\rm (Li, Ozeki, Ryj\'{a}\v{c}ek, Vr\'{a}na \cite{LORV}) } \label{LORV}
Let $\mathcal{H}$ be a $3$-hypergraph and let $W$ be the set of white vertices of its incidence graph $IG(\mathcal{H})$. The line graph $L(\mathcal{H})$ of $\mathcal{H}$ is Hamiltonian if and only if $IG(\mathcal{H})$ contains a dominating closed W-quasitrail.
\end{theorem}

\begin{figure}[!ht]
	\centering
	\includegraphics[width=140mm]{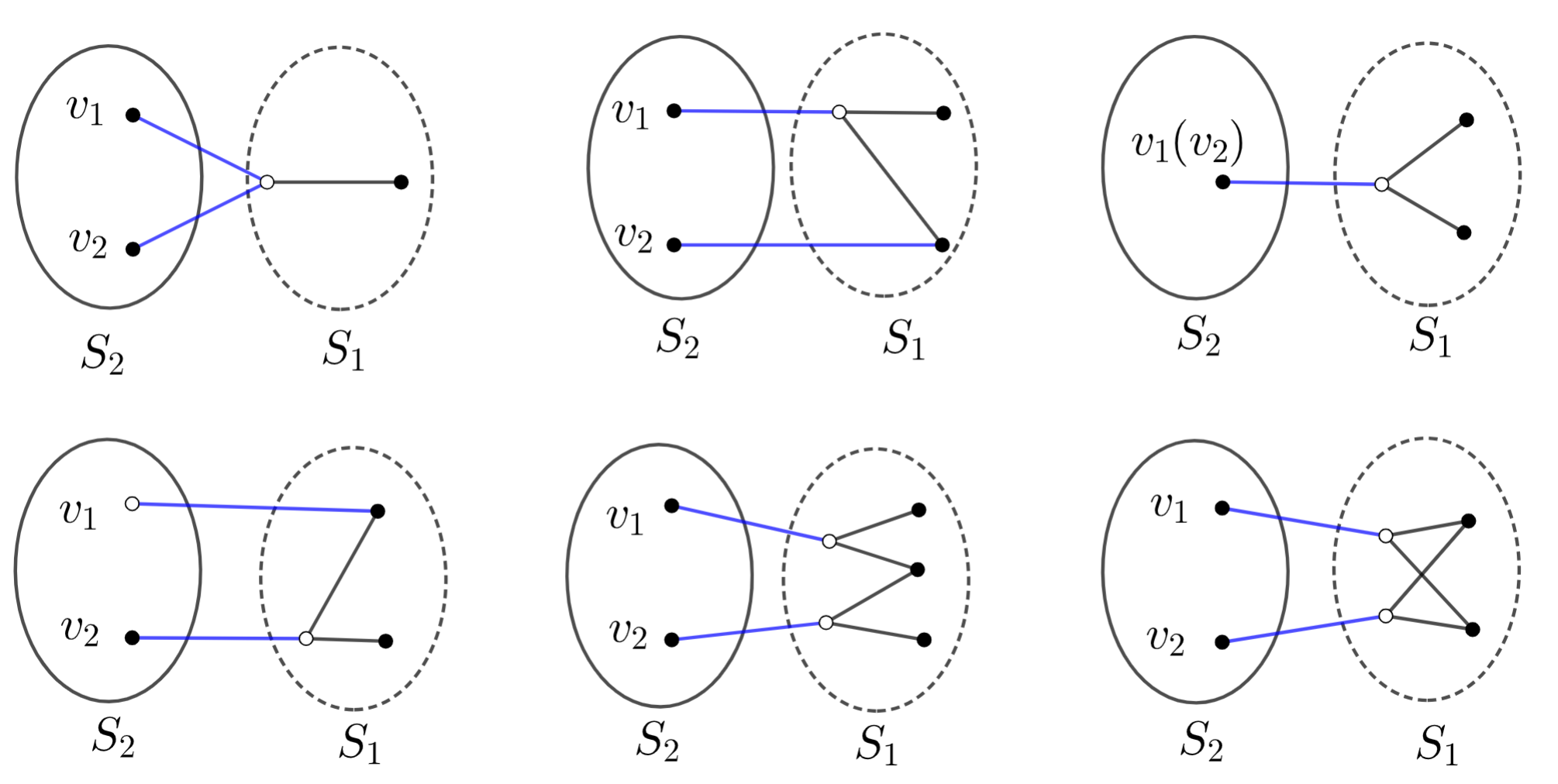}
	\caption{Non-essentially edge cut in $IG'(H)$. }
	\label{fig4}
\end{figure}

\subsection{Proof of Theorem \ref{thm3}}

Now we are ready to give a proof of Theorem \ref{thm3}.
\medskip

\noindent{\bf Proof of Theorem \ref{thm3}.} Let $\mathcal{H}$ be a 3-hypergraph whose line graph $L(\mathcal{H})$ is a $3$-connected graph with domination number at most $4$. We aim to translate the problem from the hypergraph $\mathcal{H}$ to its incidence graph $IG(\mathcal{H})$. Let $IG'(\mathcal{H})$ be the graph obtained from $IG(\mathcal{H})$ by contracting each white vertex with degree 2.  Suppose that $IG'(\mathcal{H})$ is not essentially 3-edge-connected. Then there exists a minimal edge cut $X' \subseteq E(IG'(\mathcal{H}))$ with $|X'| < 3$ such that $IG'(\mathcal{H}) - X'$ has at least two nontrivial components, say $S_1$ and $S_2$.

Note that each edge of $IG'(\mathcal{H})$ corresponds to either a hyperedge or an edge of $\mathcal{H}$. Let $X \subseteq E(\mathcal{H})$ be the set of hyperedges or edges in $\mathcal{H}$ corresponding to the edges in $X'$. Since $|X'| < 3$, we have $|X| < 3$. Now consider the hypergraph $\mathcal{H} - X$. If $\mathcal{H} - X$ has two or more nontrivial components, then in the line graph $L(\mathcal{H})$, the vertices corresponding to the hyperedges in these components would be separated by the vertex cut $X$, contradicting the assumption that $L(\mathcal{H})$ is 3-connected. Therefore, $\mathcal{H} - X$ has exactly one nontrivial component. That is to say, either $S_1$ or $S_2$ satisfies the following condition: every edge in that component corresponds to a hyperedge of $\mathcal{H}$ that belongs to $X$. Without loss of generality, assume that $S_1$ satisfies this condition. See the example in Fig. \ref{fig4}.

Let $\{v_1, v_2\} \subseteq V(X')$ be the two vertices lying in $S_2$ (See Fig. \ref{fig4}). By deleting all the vertices in $S_1 $, and adding a new edge between $v_1$ and $v_2$ (assuming that $v_1$ and $v_2$ are not adjacent), we obtain a new graph, denoted by $IG'(\mathcal{H})^1$. If $IG'(\mathcal{H})^1$ is still not essentially 3-edge-connected, we repeat this procedure. Since in each step, the number of edge cuts of size less than 3 strictly decreases, this process must terminate after finitely many steps. Let $IG'(\mathcal{H})^s$ denote the resulting graph. Then $IG'(\mathcal{H})^s$ is essentially 3-edge-connected.

For each edge $e$ that has been removed during the construction of $IG'(\mathcal{H})^s$, let $e^0$ be the edge or hyperedge that contains  $v_1$ and $v_2$. If $v_1 = v_2$, define $e^0$ to be any edge incident to $v_1$. In order to apply Theorem \ref{thm8}, we  convert the problem to the core of $IG'(\mathcal{H})^s$.  For convenience, we denote $\co(IG'(\mathcal{H})^s)$ by $\widetilde{\mathcal{H}}$ in the following discussion.

Let $\{u_1, u_2, u_3, u_4\}$ be a dominating set of $G = L(\mathcal{H})$. By the definition of the line graph of a hypergraph, there exist four edges or hyperedges $e_1, e_2, e_3, e_4$ in $\mathcal{H}$
%$\{e_1, e_2, e_3, e_4\} \subseteq E(\mathcal{H})$ 
corresponding to $u_1, u_2, u_3, u_4$, respectively, such that the set $\{e_1, \ldots, e_4\}$ dominates all edges and hyperedges of $\mathcal{H}$. Let $\{e'_1, \ldots, e'_4\}$ be the hyperedges (or edges) in $\widetilde{\mathcal{H}}$ corresponding to $\{e_1, e_2, e_3, e_4\}$. Let $\{z_1, \ldots, z_t\}$ denote the set of vertices in $\widetilde{\mathcal{H}}$ that are incident with the white vertices representing $\{e'_1, \ldots, e'_4\}$. If some $e'_i$ is an edge rather than a hyperedge, then $z_j$ refers to one of its endvertices. Thus, the vertex set $\{z_1, \ldots, z_t\} \subseteq V(\widetilde{\mathcal{H}})$ dominates all edges of $\widetilde{\mathcal{H}}$. Since each hyperedge contains three vertices and there are at most four such hyperedges, it follows that $t \leq 12$.

If there exists a closed trail $R$ in $\widetilde{\mathcal{H}}$ that passes through all of $\{z_1, \ldots, z_t\}$, then it is straightforward to check that for every construction of $e_i',\,i=1,2,3,4$, we can find a closed trail $R'$ in $IG(\mathcal{H})$. Our goal is to apply Theorem  \ref{LORV} to deduce that this trail $R'$ corresponds to a Hamilton cycle in $L(\mathcal{H})$. Let $W$ be the set of white vertices in $IG(\mathcal{H})$. We now construct a closed dominating $W$-quasitrail from $R'$.

Consider a white vertex $w$ of $IG(\mathcal{H})$ that is not visited by the trail $R'$. Since the set $\{z_1, \ldots, z_t\}$ dominates all hyperedges and deges of $\mathcal{H}$, it follows that every such white vertex $w$ is adjacent to at least one vertex $z_i \in \{z_1, \ldots, z_t\}$, and by construction, each $z_i$ is visited by $R'$. To include $w$ is in the walk, we modify $R'$ by inserting a short detour at one of the occurrences of $z_i$: namely, we replace an occurrence of $z_i$ in $R'$ with the sequence $z_i, z_iw, w, wz_i, z_i$, which adds the white vertex $w$ and the two incident edges $z_iw$ and $wz_i$ into the trail without disrupting its closed nature. By repeating this detour operation for every white vertex $w$ not appearing on $R'$ (selecting any adjacent $z_i \in \{z_1, \ldots, z_t\}$ for the detour), we obtain a closed walk in $IG(\mathcal{H})$ that visits all white vertices. Since the white vertices of $IG(\mathcal{H})$ correspond to all hyperedges and edges of $\mathcal{H}$, the resulting walk dominates all edges of $IG(\mathcal{H})$. This final walk is thus a closed dominating $W$-quasitrail in $IG(\mathcal{H})$. By Theorem \ref{LORV}, this implies that the line graph $L(\mathcal{H})$ is Hamiltonian.

Now, assuming that no closed trail $R$ exists in $\widetilde{\mathcal{H}}$ that passes through all of $\{z_1, \ldots, z_t\}$, it follows from Theorem \ref{thm8} that 
the core $\widetilde{\mathcal{H}}$ can be contracted to the Petersen graph such that the preimage of each vertex in the Petersen graph,
and hence $10 \leq t \leq 12$. Then the edge set $\{e_1, e_2, e_3, e_4\}$ must contain at least two hyperedges. %By Theorem \ref{thm8}, contains at least one vertex from $\{z_1, \ldots, z_t\}$. 
In particular, we may assume that $z_1, z_2, z_3$ are the endvertices of the hyperedge $e_1$ and are mapped to three vertices $p_1, p_2, p_3$ in the Petersen graph, respectively.

Let $w_1$ be the white vertex in $\widetilde{\mathcal{H}}$ corresponding to the hyperedge $e_1$.
By the definition of contraction, 
either 
$w_1$ forms the preimage of a vertex other than $p_1,p_2,p_3$
or
$w_1$ is contained in the preimage of $p_i$ for some $i \in \{1,2,3\}$.
For the former case,
the preimage consisting of $w_1$ contains no vertex in $\{z_1, \ldots, z_t\}$,
a contradiction.
Thus, we may assume the latter occurs.

By the symmetry, we may assume that 
$w_1$ is contained in the preimage of $p_1$,
and let $W_1 \subseteq V(\widetilde{\mathcal{H}})$ be the preimage of $p_1$.
Since $z_2$ and $z_3$ are adjacent to $w_1$ in $\widetilde{\mathcal{H}}$,
$p_2$ and $p_3$ are the neighbors of $p_1$ in the Petersen graph.
Let $p_4$ be the unique neighbor of $p_1$ 
%in the Petersen graph,
other than $p_2$ and $p_3$.
Thus, there exists an edge $f_4$ in $\mathcal{H}$ 
corresponding to the edge $p_1p_4$ in the Petersen graph.
However,
$\{w_1,f_4\}$ forms an essential $2$-edge-cut in $\mathcal{H}$
that corresponds to a $2$-cut in $G = L(\mathcal{H})$,
a contradiction.
%
\if0
the vertices $p_1, p_2, p_3$ must either form a path of length 3 or an independent set adjacent to a common vertex $p_4$. The first case is impossible: if $p_1, p_2, p_3$ form a path of length $3$ in the Petersen graph, since $L(\mathcal{H})$ is 3-connected, the degree of $z_2$  in $\widetilde{\mathcal{H}}$ is at most $2$, contradicting the assumption that  $\widetilde{\mathcal{H}}$ is the core of $IG'(\mathcal{H})^s$, Therefore, only the second case can occur.

Since every preimage of the Petersen graph must contain at least one endvertex of $\{e_1,e_2, e_3, e_4\}$, we have that the preimage of $p_4$, the common neighbor of $p_1, p_2, p_3$, must contain a hyperedge among $\{e_2, e_3, e_4\}$. However, since the images of three vertices in hyperedge cannot form a path, there exists a vertex in the Petersen graph whose preimage does not contain an endvertex of $\{e_1, e_2, e_3,e_4\}$, a contradiction.
\fi
This completes the proof.\hfill $\Box$

\subsection{Best possibility of Theorem \ref{thm3}}

In Theorem \ref{thm3}, we established that  every $3$-connected line graph of a $3$-hypergraph with domination number at most $4$ is Hamiltonian. This result is sharp: for instance, the line graph of a graph obtained from the Petersen graph by attaching at least one pendant edge to each of its vertices is a counterexample when the domination number equals 5, see Theorem \ref{thm1}.

A {\it clique covering} of the graph $G$ is a family $\mathcal{K}$ of cliques of $G$ such that every vertex and every edge of $G$ belongs to some clique from $\mathcal{K}$. The following result characterizes graphs that are line graphs of an $r$-hypergraph.

\begin{theorem}{\rm (Berge \cite{B} and Skums,  Suzdal,  Tyshkevich \cite{SST})\label{thm11}}
For every integer $r \geq 2$, a graph $G$ is a line graph of an $r$-hypergraph if and only if $E(G)$ can be covered by a system of cliques $\mathcal{K}$ such that every vertex of $G$ is in at most $r$ cliques of $\mathcal{K} $.
\end{theorem}

\begin{figure}[!ht]
	\centering
	\includegraphics[width=100mm]{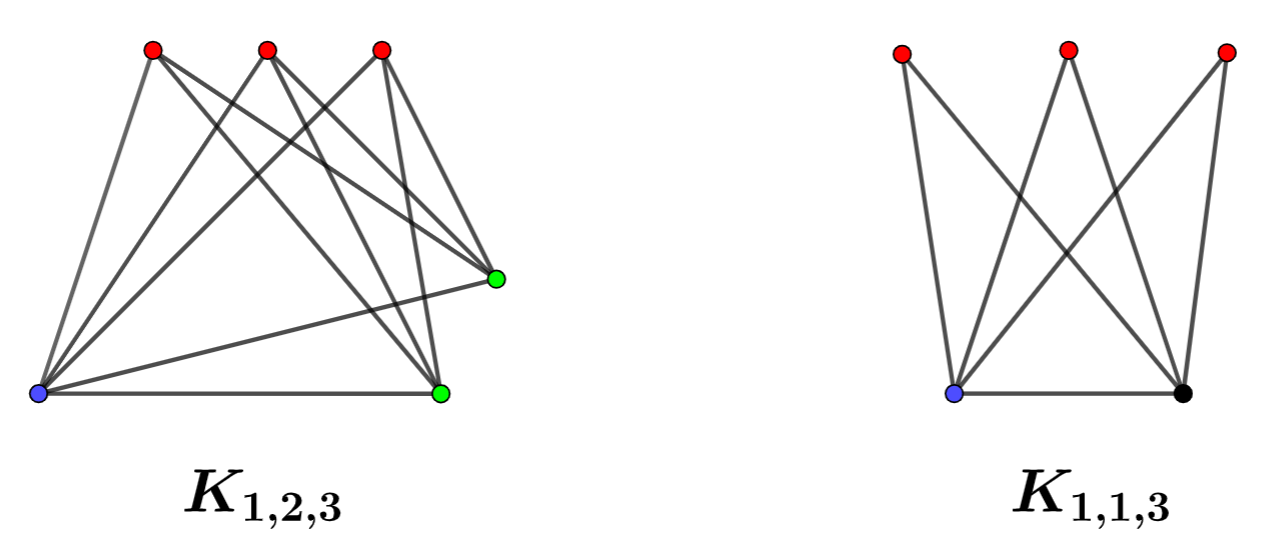}
	\caption{The graphs $K_{1,2,3}$ and $K_{1,1,3}$. }
	\label{fig5}
\end{figure}

Let $K_{1,2,3}$ denote the complete tripartite graph with partite sets of sizes $1$, $2$, and $3$, respectively (see Figure~\ref{fig5}). It is easy to verify that $K_{1,2,3}$ is 3-connected but not Hamilton-connected. By Theorem~\ref{thm11}, $K_{1,2,3}$ is the line graph of a 3-hypergraph. This naturally raises the question of how Hamilton-connectedness is related to the domination number in 3-connected line graphs of 3-hypergraphs. The example of $K_{1,2,3}$ shows that even when the domination number is $1$, the 3-connected line graph of a  3-hypergraph may fail to be Hamilton-connected.
Moreover, by replacing the red vertex in $K_{1,2,3}$ (see Figure~\ref{fig5}) with a clique and connecting every vertex of this clique to all vertices in the size-1 and size-2 parts of $K_{1,2,3}$, one can construct infinitely many larger counterexamples.  Similarly, the graph $K_{1,1,3}$ illustrates that a 2-connected line graph of a 3-hypergraph with domination number $1$ may not be Hamiltonian. Again, replacing the red vertex in $K_{1,1,3}$ with a clique yields an infinite family of such non-Hamiltonian graphs.

\section*{Acknowledgement} Kenta Ozeki was supported by JSPS KAKENHI, Grant Numbers 22K19773 and 23K03195. Leilei Zhang's work was supported by JSPS KAKENHI,  Grant Number 25KF0036, the NSF of Hubei Province Grant Number 2025AFB309,  the China Postdoctoral Science Foundation  Grant Number 2025M773113, the Fundamental Research Funds for the Central Universities, Central China Normal University Grant Number CCNU24XJ026.
	
\section*{Declaration}
	
\noindent$\textbf{Conflict~of~interest}$
The authors declare that they have no known competing financial interests or personal relationships that could have appeared to influence the work reported in this paper.
	
\noindent$\textbf{Data~availability}$
Data sharing not applicable to this paper as no datasets were generated or analysed during the current study.

\end{document}